\documentclass[12pt,a4paper]{amsart}
\usepackage[all]{xy}
\usepackage{amssymb}
\newtheorem{theorem}{Theorem}

\newtheorem{prop}[theorem]{Proposition}

\newtheorem{definition}[theorem]{Definition}
\theoremstyle{remark}
\newtheorem{example}[theorem]{Example}
\newtheorem{remark}[theorem]{Remark}

\newtheorem{problem*}{Problem}
\newtheorem{remark*}{Remark}
\newtheorem{convention*}{Convention}
\newtheorem{notation*}{Notation}
\newtheorem{examples*}{Examples}
\newtheorem{example*}{Example}
\newtheorem{warning*}{Warning}

\def\R{{\mathbb R}}

\def\Z{{\mathbb Z}}

\def\H{{\mathcal H}}

\begin{document}
\title[Generalized low-pass filters and multiresolution analyses]{Generalized filters, the low-pass condition,
and connections to multiresolution analyses}
\author[L. W. Baggett]{Lawrence~W.~Baggett}
\address{Lawrence Baggett, Department of Mathematics, University of Colorado, Boulder, Colorado 80309, USA}
\email{baggett@euclid.colorado.edu}
\author[V. Furst]{Veronika~Furst}
\address{Veronika Furst, Department of Mathematics, Fort Lewis College, Durango, Colorado 81301, USA}
\email{furst\_v@fortlewis.edu}
\author[K. D. Merrill]{Kathy~D.~Merrill}
\address{Kathy Merrill, Department of Mathematics, Colorado College, Colorado Springs, Colorado, 80903, USA}
\email{kmerrill@coloradocollege.edu}
\author[J. A. Packer]{Judith~A.~Packer}
\address{Judith Packer, Department of Mathematics, University of Colorado, Boulder, Colorado 80309, USA}\email{packer@euclid.colorado.edu}
\begin{abstract}
We study generalized filters that are associated to multiplicity functions
and homomorphisms of the dual of an abelian group.
These notions are based on the structure of generalized
multiresolution analyses.  We investigate when the Ruelle operator corresponding to such a filter
is a pure isometry, and then use that characterization
to study the problem of when
a collection of closed subspaces, which satisfies all the conditions
of a GMRA except the trivial intersection condition,
must in fact have a trivial intersection.
In this context, we obtain a generalization of a theorem of Bownik and Rzeszotnik.
\end{abstract}

\maketitle
\section{Introduction}
Filters have historically been an essential tool used in both building and analyzing wavelets and multiresolution structures.  In particular, filters traditionally called ``low-pass" arise naturally from refinement equations for multiresolution analyses (MRAs) and generalized multiresolution analyses (GMRAs).   Beginning with work of Mallat \cite{Ma} and Meyer  \cite{Me},  the process of defining filters from a multiresolution structure was also reversed; that is, functions that behave like low-pass filters have been used to build the structures.  This construction technique has been remarkably fruitful, producing, for example, the smooth and well-localized wavelets of Daubechies \cite{Da}.  In generalizing this procedure to allow less restrictive conditions on the filters as well as on the setting, for example in  \cite{BJMP}, \cite{IJKLN}, and \cite{BJ}, properties of an operator associated with the filter, called a Ruelle operator, are used to justify this construction.  The essential ingredient is that the Ruelle operator be a pure isometry.  A theorem giving general conditions under which the Ruelle operator is a pure isometry in the case of an integer dilation in $L^2(\mathbb T)$ appeared in \cite{BJ}. 

In this paper we derive a similar theorem (Theorem \ref{eigenvalue} in Section \ref{pure}) in a quite general context.  We then exploit this theorem both in analyzing multiresolution structures and in building them.  Our central result of the first type addresses the question of when a structure that satisfies all the properties of a GMRA except possibly the trivial intersection property, must satisfy that as well.  This generalizes work of Bownik and collaborators (\cite{Bow}, \cite{BR}).  Our main result of the second type is to show that very little in the way of a low-pass condition is needed when building GMRAs from filters using direct limits as in \cite{IJKLN} and \cite{AIJLN}.  

Our general context is as follows:  
Let $\Gamma$ be a countable abelian group (written additively) with dual group
$\widehat\Gamma$ (written multiplicatively), equipped with Haar measure $\mu$ (of total mass 1).  
  Let $\alpha$ be an isomorphism of $\Gamma$ into itself, and suppose that the index of $\alpha(\Gamma)$ in $\Gamma$ equals $N>1.$
Assume further that $\cap_{n\geq 0}\,\alpha^n(\Gamma) = \{0\}.$
Write $\alpha^*$ for the dual endomorphism of $\widehat\Gamma$ onto itself defined by
$[\alpha^*(\omega)](\gamma) = \omega(\alpha(\gamma)),$ and note that
the kernel of $\alpha^*$ contains exactly $N$ elements
and that $\alpha^*$ is ergodic with respect to the Haar measure  on $\widehat\Gamma.$
Write $K = \cup_{n> 0} \ker({\alpha^*}^n),$
and note that, because $\cap_{n\geq 0}\,\alpha^n(\Gamma) = \{0\},$
 $K$ is dense in $\widehat\Gamma.$  
\par
Of course the standard example (e.g., from wavelet theory) of these ingredients is
where $\Gamma=\mathbb Z,$ $\widehat\Gamma=\mathbb T,$ and $\alpha(k)=2k.$
Or, more generally, $\Gamma=\mathbb Z^d,$ and $\alpha(\vec x) = A\vec x,$
where $A$ is a $d\times d$ integer dilation matrix of determinant $N.$
\par
Let $m:\widehat\Gamma \to \{0,1,2,\ldots,\infty\}$ be a Borel map into the set of nonnegative integers union $\infty,$ and write $\sigma_i$ for $\{\omega\in\widehat\Gamma : m(\omega)\geq i\}.$
Note that 
\begin{equation*}\label{m-equation}
m(\omega) = \sum_i \chi_{\sigma_i}(\omega).
\end{equation*}
We remark that such functions $m$ arise, via Stone's Theorem on unitary representations of abelian groups, as multiplicity functions
associated to such representations of $\Gamma,$
and we will in fact invoke this relationship in the final section.  In that context, we will make use of 
a unitary representation $\pi$ of $\Gamma$,  
acting in a Hilbert space $\mathcal H,$ and a unitary operator $\delta$ 
 on $\mathcal H$ for which
\[
\delta^{-1}\pi_\gamma \delta = \pi_{\alpha(\gamma)}
\]
for all $\gamma\in \Gamma.$
\par
In this general setting, we define a filter as follows:  

\begin{definition}
\label{filter}
A (possibly infinite) matrix $H=[h_{i,j}]$ of Borel, complex-valued functions
on $\widehat\Gamma$ is called a {\bf filter} relative to $m$ and $\alpha^*$ if,
for every $j,$ $h_{i,j}$ is supported in $\sigma_j,$ and $H$
satisfies the following ``filter equation:''
\begin{equation}\label{filterequation}
\sum_{\alpha^*(\zeta)=1} \sum_j h_{i,j}(\omega\zeta)\overline{h_{i',j}(\omega\zeta)}
= N\delta_{i,i'}\chi_{\sigma_i}(\alpha^*(\omega))
\end{equation}
for almost all $\omega\in\widehat\Gamma.$
\end{definition}
In the standard situation described above, i.e., where $\Gamma=\mathbb Z,$
$\alpha(k) = 2k,$ and where $m$ is the identically 1 function,
a filter relative to $m$ and $\alpha^*$ is just a $1\times1$ matrix (function) $h,$ and the
filter equation becomes
\[
|h(z)|^2 + |h(-z)|^2 = 2,
\]
for almost all $z\in\mathbb T$, which is the 
classical equation satisfied by a quadrature mirror filter.
These are the filters that played a central role in the
early theory of multiresolution analyses and wavelets in $L^2(\mathbb R).$
Indeed, in the classical case, where $\phi$ is a scaling function
for an MRA in $L^2({\mathbb R}),$ we know that
the integral translates $T_n(\phi) = \phi(\cdot-n)$ form an orthonormal basis for the core subspace $V_0,$
and we may define a unitary operator $J$ from $V_0$ onto $L^2(\mathbb T)$
by sending the basis vector $T_n(\phi)$ to the function $z^n.
$  This correspondence between
an orthonormal basis of $V_0$ with the canonical Fourier basis for $L^2(\mathbb T)$
is clearly a unitary operator.
  Furthermore, $J$ sends the
element 
$\phi(x/2)/\sqrt2  = \sum c_n T_n(\phi)$
to the function $\sum_n c_nz^n = h(z),$
where $h$ is the associated quadrature mirror filter.
We notice that, in addition to the fact that $h$ satisfies the
quadrature mirror equation,
it satisfies another condition.
Namely, if $\delta$ denotes the dilation operator on $L^2(\mathbb R)$ given by
$[\delta(f)](x) = \sqrt2 f(2x),$
then one can verify that the operator $J\circ \delta^{-1}\circ J^{-1}$ on $L^2(\mathbb T)$ is given by
\[
[[J\circ \delta^{-1}\circ J^{-1}](f)](z) = h(z) f(z^2) = [S_h(f)](z)
\]
for every $f\in L^2(\mathbb T).
$
We will call such an operator $S_h$ a {\em Ruelle operator.
}
Because $\delta^{-1}$ is an isometry on $V_0,$ and
$\cap \text{ Range}(\delta^{-n}) = \cap V_n = \{0\},$
it follows that the operator $S_h$ has these same properties.
That is, $S_h$ is a ``pure isometry.''
\par
In the next section we define a Ruelle operator $S_H$ similarly associated with an abstract filter $H$ as in Definition \ref{filter}, and present our first main result, a characterization of when
this Ruelle operator is a pure isometry.  The final two sections contain the applications of this result.  
\section{Filters and pure isometries}
\label{pure}
Let $H$ be a filter relative to  $m$ and $\alpha^*$.   Whenever the formula
\begin{align*}
[S_H(f)](\omega) &= H^t(\omega)f(\alpha^*(\omega)) \cr
& = \bigoplus_j \sum_i H_{i,j}(\omega) f_i(\alpha^*(\omega)),
\end{align*}
defines a bounded operator from $\bigoplus_i L^2(\sigma_i,\mu)$ into itself, we will call $S_H$ the Ruelle operator associated to $H$.  In the contexts we study in this paper, this will always be the case.  

We now prove a generalization of the filter equation of Definition \ref{filter} that will provide a crucial step in determining when the Ruelle operator $S_H$ is an isometry.  If the function $m$  associated to $H$ is finite a.e., this proposition follows from the standard filter equation by induction (see Lemma 9 in \cite{IJKLN}).  However, without this restriction, it requires a more careful argument exploiting the fact that, in the situations we study, $S_H$ is an isometry.  
 \begin{prop}\label{superfilter}
Let $H$ be a filter relative to $m$ and $\alpha^*$, and assume that the associated Ruelle operator $S_H$ is an isometry.  Then:
\begin{equation*}
 \frac1{N^n} \sum_{{\alpha^*}^n(\zeta)=1} \sum_i
\left [\prod_{k=0}^{n-1} H^t({\alpha^*}^k(\omega\zeta))\right]_{i,j}
\overline{\left[\prod_{k'=0}^{n-1}H^t({\alpha^*}^{k'}(\omega\zeta))\right]}_{i,j'}
= \delta_{j,j'}\chi_{\sigma_j}(\alpha^{*n}(\omega)).
\end{equation*}
\end{prop}
\begin{proof}
In the calculations below, certain sums and integrals
have to be exchanged.  These exchanges are justified by
Fubini's Theorem, where we rely several times on the fact that
for any $f\in \bigoplus_j L^2(\sigma_j,\mu)$
the element $S_H^n(f)$ is again in
$\bigoplus_j L^2(\sigma_j,\mu),$ and therefore
$\sum_j |[S_H^n(f)]_j(\omega)|^2$ is finite for almost every $\omega.$
\par
For each $f$ and $g$ in $\bigoplus_j L^2(\sigma_j,\mu),$ we have
\begin{align*}
& \sum_i \int_{\widehat\Gamma} \sum_j \left[\prod_{k=0}^{n-1} H^t({\alpha^*}^k(\omega))\right]_{i,j} f_j({\alpha^*}^n(\omega)) 
\overline{\sum_{j'} \left[\prod_{k'=0}^{n-1} H^t({\alpha^*}^{k'}(\omega))\right]_{i,j'} g_{j'}({\alpha^*}^n(\omega))  }\, d\omega \cr
& = \sum_i \int_{\widehat\Gamma} [S_H^n(f)]_i(\omega)
\overline{[S_H^n(g)]_i(\omega)} \, d\omega \cr
& = \langle S_H^n(f) \mid S_H^n(g)\rangle \cr
& = \langle f\mid g\rangle \cr
& = \sum_j \int_{\widehat\Gamma} f_j(\omega) \overline{g_j(\omega)}\, d\omega \cr
& = \sum_j \int_{\widehat\Gamma} f_j({\alpha^*}^n(\omega)) \overline{g_j({\alpha^*}^n(\omega))} \, d\omega.
\end{align*}
Therefore,
\begin{align*}
& \frac{1}{N^n} \sum_{{\alpha^*}^n(\zeta)=1}  \sum_i\int_{\widehat\Gamma} \left(\sum_j \left[\prod_{k=0}^{n-1} H^t({\alpha^*}^k(\omega\zeta))\right]_{i,j} f_j({\alpha^*}^n(\omega))\right) \cr
& \qquad \qquad\qquad\qquad\quad\quad\quad\quad\quad\cdot \left(
\overline{\sum_{j'} \left[\prod_{k'=0}^{n-1} H^t({\alpha^*}^{k'}(\omega\zeta))\right]_{i,j'} g_{j'}({\alpha^*}^n(\omega))}\right)  \, d\omega\\
&= \sum_j \int_{\widehat\Gamma} f_j({\alpha^*}^n(\omega)) \overline{g_j({\alpha^*}^n(\omega))} \, d\omega.
\end{align*}
Write $C_j$ for the element of the direct sum space
$\bigoplus_j L^2(\sigma_j,\mu)$ whose $j$th coordinate is $\chi_{\sigma_j}$
and whose other coordinates are 0.  Set $f=\chi_{_E} C_j,$ for $E\subseteq \sigma_j,$
and $g= \chi_{_{E'}}C_{j'},$ for $E'\subseteq \sigma_{j'}.$  Then we have
\begin{align*}
& \frac1{N^n} \sum_\zeta \sum_i \int_{{\alpha^*}^{-n}(E\cap E')}
\left[\prod_{k=0}^{n-1} H^t({\alpha^*}^k(\omega\zeta))\right]_{i,j}
\overline{\left[\prod_{k'=0}^{n-1} H^t({\alpha^*}^{k'}(\omega\zeta))\right]_{i,j'}} \, d\omega\cr
& = \frac1{N^n} \sum_\zeta \sum_i \int_{\widehat\Gamma}
\left[\prod_{k=0}^{n-1} H^t({\alpha^*}^k(\omega\zeta))\right]_{i,j} \chi_{_E}({\alpha^*}^n(\omega))C_j({\alpha^*}^n(\omega)) \cr
&\qquad\qquad\qquad \qquad \qquad \cdot \ \overline{\left[\prod_{k'=0}^{n-1} H^t({\alpha^*}^{k'}(\omega\zeta))\right]_{i,j'}} \chi_{_{E'}}({\alpha^*}^n(\omega)) C_{j'}({\alpha^*}^n(\omega)) \, d\omega  \cr
& = \delta_{j,j'} \langle \chi_{{{\alpha^*}^{-n}(E)}}C_j \mid \chi_{_{{\alpha^*}^{-n}(E')}}C_{j'} \rangle \cr
& = \delta_{j,j'} \int_{\widehat\Gamma} \chi_{{E\cap E'}}({\alpha^*}^n(\omega))\, d\omega \cr
& = \delta_{j,j'} \int_{{\alpha^*}^{-n}(E\cap E')} 1 \, d\omega.
\end{align*}
Since this is true for any Borel sets $E$ and $E',$ the proposition follows.
\end{proof}

We now prove our first main result, establishing conditions under which a Ruelle operator $S_H$ that is an isometry must in fact be a pure isometry.  This theorem generalizes Theorem 3.1 in \cite{BJ}, which finds a similar conclusion in the setting of integer dilations in $L^2(\mathbb T).$ 

\begin{theorem}
\label{eigenvalue}
Assume that $m$ is finite on a set of positive measure.
If $S_H$ is an isometry on $\bigoplus_j L^2(\sigma_j,\mu),$
then $S_H$ fails to be a pure isometry if and only if
it has an eigenvector.
Specifically, $S_H$ fails to be a pure isometry if and only if
there exists a nonzero element $f\in \bigoplus_j L^2(\sigma_j,\mu),$ and a scalar $\lambda$ of absolute value 1,
such that $S_H(f) = \lambda f.
$
Moreover, if $f$ is a unit eigenvector for $S_H$, then $\|f(\omega)\|=1$ a.e.
\end{theorem}
\begin{proof}
Write $R_n$ for the range of the isometry $S_H^n,$ and write
$R_\infty$ for the intersection $\cap R_n$ of the $R_n$'s.
By definition, $S_H$ is a {\bf pure isometry\/} if and only if $R_\infty = \{0\}.
$
\par
If $S_H$ has an eigenfunction $f,$ say $S_H(f)=\lambda f,$ with $\lambda\neq 0,$
then clearly $f$ belongs to the range of each operator $S_H^n,$
and hence $f\in R_\infty.
$
Therefore, $R_\infty\neq \{0\},$ and $S_H$ is not a pure isometry.
\par
Conversely, suppose $S_H$ is not a pure isometry.
We now adapt an argument in \cite{IJKLN}
that was based on the reverse martingale convergence theorem
(See Theorem 10.6.1 in \cite{dud}.)
For each $n\geq 1$, let $\mathcal M_n$ be the $\sigma$-algebra of Borel subsets of $\widehat\Gamma$ that are invariant
under multiplication by elements in the kernel of ${\alpha^*}^n.$
Let $f$ and $g$ be two nonzero vectors in $R_\infty,  $
and define a sequence of random variables $\{X_n\} \equiv \{X^{f,g}_n\}$ on $\widehat\Gamma$ by
\[
X_n(\omega) = \frac1{N^n} \sum_{{\alpha^*}^n(\zeta) = 1} \langle f(\omega\zeta) \mid g(\omega\zeta) \rangle.
\]
Then it follows directly that $X_n$ is $\mathcal M_n$-measurable, and
the conditional expectation of $X_n,$ given $\mathcal M_{n+1},$ equals $X_{n+1}.
$
Therefore, the sequence $\{X_n,\mathcal M_n\}$ is an integrable, reverse martingale.
Hence, using the reverse martingale convergence theorem, we have that the sequence
$\{X_n(\omega)\}$ converges almost everywhere and in $L^1$ norm to an integrable function $L$ on $\widehat\Gamma.$
\par
Clearly, $L(\omega\zeta) = L(\omega)$ for almost every $\omega$
and every $\zeta \in K=\cup_{n>0}\ker({\alpha^*}^n).$
Hence, the Fourier coefficient $c_\gamma(L)$ satisfies
$c_\gamma(L) = \gamma(\zeta)c_\gamma(L)$ for every
$\zeta\in K,$
implying that $c_\gamma(L)=0$ unless $\gamma(\zeta) = 1$
for all $\zeta\in K.$
Since $K$ is dense in $\widehat\Gamma,$ it then follows that $c_\gamma(L)=0$
for all $\gamma$ except $\gamma=0.$
Consequently, $L(\eta)$ is a constant function, and we have, from the $L^1$ convergence of the sequence $\{X_n\}$,
\begin{align*}
L(\eta) & = \int_{\widehat\Gamma} L(\omega)\, d\omega  \cr
& = \lim_{n\geq 1} \int_{\widehat\Gamma} X_n(\omega)\, d\omega \cr
& = \lim_{n\geq 1} \frac1{N^n} \sum_{{\alpha^*}^n(\zeta) = 1} \int_{\widehat\Gamma} \langle f(\omega\zeta) \mid g(\omega\zeta) \rangle\, d\omega \cr
& = \int_{\widehat\Gamma} \langle f(\omega) \mid g(\omega)\rangle \, d\omega \cr
& = \langle f \mid g \rangle.
\end{align*}
Therefore, the reverse martingale $X_n$
converges almost everywhere to the constant $\langle f \mid g\rangle.$
\par
For each $\omega,$ write $N_\omega$ for the set of all natural numbers $n$ for which $m({\alpha^*}^n(\omega)) < \infty.$
Since $m$ is finite on a set of positive measure, the ergodicity of $\alpha^*$ implies that $N_\omega$ is infinite for almost all $\omega.$
We show next that, for each $n\in N_\omega,$ there is a
different expression for $X_n(\omega).$
To wit, for each $n\in N_\omega ,$ define $f_n = {S_H^*}^n(f)$
and $g_n = {S_H^*}^n(g).
$
Since $S_H$ is a unitary operator on $R_{\infty}$, we have
\begin{align*}
& X_n(\omega)
= \frac1{N^n} \sum_{{\alpha^*}^n(\zeta) = 1} \langle f(\omega\zeta) \mid g(\omega\zeta)\rangle\\
& = \frac1{N^n} \sum_\zeta \left\langle  \prod_{k=0}^{n-1} H^t({\alpha^*}^k(\omega\zeta))f_n({\alpha^*}^n(\omega)) \mid
  \prod_{k'=0}^{n-1} H^t({\alpha^*}^{k'}(\omega\zeta))g_n({\alpha^*}^n(\omega)) \right\rangle\\
& = \frac1{N^n} \sum_\zeta \sum_j
\sum_i \left[\prod_{k=0}^{n-1} H^t({\alpha^*}^k(\omega\zeta))\right]_{j,i}
\overline{\sum_{i'}\left[\prod_{k'=0}^{n-1}H^t({\alpha^*}^{k'}(\omega\zeta))\right]_{j,i'}}
{f_n}_i({\alpha^*}^n(\omega)) \overline{{g_n}_{i'}({\alpha^*}^n(\omega))}.
\end{align*}
When interchanging the sums in the previous expression is justified,
we may continue this computation; then, using Proposition \ref{superfilter}, we would obtain
\begin{align*}
& \frac1{N^n} \sum_i \sum_{i'} {f_n}_i({\alpha^*}^n(\omega)) \overline{{g_n}_{i'}({\alpha^*}^n(\omega))} \sum_{\zeta} \sum_j \left[\prod_{k=0}^{n-1} H^t({\alpha^*}^k(\omega\zeta))\right]_{j,i}
\overline{\left[\prod_{k'=0}^{n-1}H^t({\alpha^*}^{k'}(\omega\zeta))\right]_{j,i'}} \cr
& = \sum_i {f_n}_i({\alpha^*}^n(\omega)) \overline{{g_n}_i({\alpha^*}^n(\omega))} \cr
& = \langle f_n({\alpha^*}^n(\omega)) \mid g_n({\alpha^*}^n(\omega)) \rangle.
\end{align*}
This gives the different expression for $X_n(\omega)$ that we want, whenever we can justify the interchanges of sums in the previous computations:
\begin{equation} \label{alternateexpression}
X_n(\omega) = \langle f_n({\alpha^*}^n(\omega)) \mid g_n({\alpha^*}^n(\omega)) \rangle.
\end{equation}
\par
The following calculation, which again uses Proposition \ref{superfilter}
and the Cauchy-Schwarz Inequality, shows that
the interchange of sums above is justified
whenever the sums on $i$ and $i'$ are finite sums.
Because of Proposition \ref{superfilter}, the sums on $i$ and $i'$ will be finite if
$m({\alpha^*}^n(\omega)) < \infty,$
and this is the case when $n\in N_\omega.$
Hence, the computation below will complete the derivation of
Equation (\ref{alternateexpression}).
Note also that the sums on $i$ and $i'$ will be finite sums if the vectors $f_n({\alpha^*}^n(\omega))$ and $g_n({\alpha^*}^n(\omega))$
only have a finite number of nonzero coordinates.
We will use this later on.
\begin{align*}
&\frac1{N^n} \sum_{i=1}^c \sum_{i'=1}^{c'} \sum_\zeta \sum_j
 \left| \left[\prod_{k=0}^{n-1} H^t({\alpha^*}^k(\omega\zeta))\right]_{j,i}
\overline{\left[\prod_{k'=0}^{n-1}H^t({\alpha^*}^{k'}(\omega\zeta))\right]_{j,i'}}
{f_n}_i({\alpha^*}^n(\omega)) \overline{{g_n}_{i'}({\alpha^*}^n(\omega))} \right| \cr
& \leq \frac1{N^n} \sum_{i=1}^c \sum_{i'=1}^{c'}
|{f_n}_i({\alpha^*}^n(\omega)) {g_n}_{i'}({\alpha^*}^n(\omega))| \cr
& \quad \quad \cdot \ \left( \sum_{\zeta, j} \left| \left[\prod_{k=0}^{n-1} H^t({\alpha^*}^k(\omega\zeta))\right]_{j,i} \right|^2
 \sum_{\tilde\zeta, \tilde j} \left| \left[\prod_{k'=0}^{n-1}H^t({\alpha^*}^{k'}(\omega\tilde\zeta))\right]_{\tilde j,i'} \right|^2 \right)^{1/2} \cr 
& = \sum_{i=1}^c  |{f_n}_i({\alpha^*}^n(\omega)) \chi_{\sigma_i}({\alpha^*}^n(\omega))| 
\sum_{i'=1}^{c'} |{g_n}_{i'}({\alpha^*}^n(\omega)) \chi_{\sigma_{i'}}({\alpha^*}^n(\omega))| \cr
& \leq \left( \sum_{i=1}^c  |{f_n}_i({\alpha^*}^n(\omega))|^2 \sum_{\tilde i=1}^c |\chi_{\sigma_{\tilde i}}({\alpha^*}^n(\omega))|^2 \right)^{1/2}
 \left( \sum_{i'=1}^{c'} |{g_n}_{i'}({\alpha^*}^n(\omega))|^2 \sum_{\tilde i'=1}^{c'} |\chi_{\sigma_{\tilde i'}}({\alpha^*}^n(\omega))|^2 \right)^{1/2} \cr
& \leq \sqrt{c c'} \|f_n({\alpha^*}^n(\omega))\| \|g_n({\alpha^*}^n(\omega))\| \cr 
& <  \infty,
\end{align*}
for almost every $\omega.$
\par
The first conclusion we can draw from Equation (\ref{alternateexpression}) is that for almost all $\omega$,
\begin{align*}
\lim_{n\in N_\omega} \langle f_n({\alpha^*}^n(\omega)) \mid g_n({\alpha^*}^n(\omega)) \rangle 
& = \lim_{n\to\infty} X_{n}(\omega) \cr
&= \langle f\mid g \rangle,
\end{align*}
or, setting $g=f,$ for $f$ a unit vector in $R_\infty,$
\[
\lim_{n\in N_\omega} \|f_{n}({\alpha^*}^{n}(\omega))\| = \|f\| =1.
\]
\par
A second conclusion we may draw is that
we must have $\sigma_1=\widehat\Gamma,$ i.e., $m(\omega)\geq 1$ a.e.
Indeed, if $m(\omega)=0$ for all $\omega$ in a set $F$
of positive Haar measure, then from the ergodicity of $\alpha^*$,
we must have ${\alpha^*}^n(\omega) \in F$ infinitely often for almost all $\omega,$
so that $\|f_n({\alpha^*}^n(\omega))\| = 0$ infinitely often.
But, since each such integer $n$ belongs to $N_\omega,$ this contradicts the first claim above.
\par
Now let $i_0$ satisfy $\sigma_{i_0}=\widehat\Gamma$ and $\sigma_{i_0+1}$ be a proper subset of $\widehat\Gamma$ of measure strictly less than 1.
(Of course $\sigma_{i_0+1}$ could be the empty set, if $m(\omega) \equiv i_0.
$)
Then, by a similar kind of ergodicity argument as was used above, we know that
for almost all $\omega,$ and for infinitely many values of $n$,
$[f({\alpha^*}^n(\omega))]_i = 0$ for all $i>i_0$
and all $f\in R_{\infty}.
$
Indeed, this is true whenever ${\alpha^*}^n(\omega) \notin \sigma_{i_0+1},$
and this occurs infinitely often for almost all $\omega.$
Moreover, each such $n$ belongs to $N_\omega.$
\par
Let $f^1,\ldots,f^k$ be orthonormal vectors in $R_\infty.
$
Then, for infinitely many sufficiently large $n,$ we must have that the $k$ $i_0$-dimensional vectors
\[
\{[f^p_{n}({\alpha^*}^{n}(\omega))]_1, \ldots,[f^p_{n}({\alpha^*}^{n}(\omega))]_{i_0}\}
\]
are nearly orthogonal and nearly of unit length.
Consequently, $k$ must be $\leq i_0.
$
Hence $R_\infty$ is finite dimensional, and therefore $S_H$ (a unitary operator on $R_\infty$) must have an eigenvector.
\par
To prove the final part of the proposition, let $f$ be a unit vector in $R_\infty.$
From the second claim above, we know that 
the coordinates $f_i$ of $f$
are all 0 for $i>i_0.$
Therefore, the interchanges of summations in the calculations above
are justified, and we obtain
\begin{equation*}
X^{f,f}_n(\omega)
 = \|f_n({\alpha^*}^n(\omega))\|^2,
\end{equation*}
so that
\[
\lim_{n\to \infty} \|f_n({\alpha^*}^n(\omega))\|^2 = \|f\|^2 = 1
\]
for almost all $\omega.$
\par
Finally, let $f$ be a unit eigenvector for $S_H$.  We have then that
\begin{align*}
\lim_{n\to\infty} \|f({\alpha^*}^n(\omega))\|^2
 & = \lim_{n\to\infty}  \|[S_H^n(f_n)]({\alpha^*}^n(\omega))\|^2 \cr
& = \lim_{n\to\infty}  \|f_n({\alpha^*}^n(\omega))\|^2 \cr
& =1.
\end{align*}
By the ergodicity of $\alpha^*,$ it follows that
$\|f(\omega)\| =1$ almost everywhere.
\end{proof}

\section{Pure isometries and the low pass condition}
\label{lowpass}

In this section, we use Theorem \ref{eigenvalue} to eliminate the need for a restrictive low-pass condition when building GMRAs from filters via the direct limit construction of \cite{IJKLN} and \cite{AIJLN}.
First we recall the definition:   
\begin{definition}
\label{GMRA}
A collection $\{V_j\}_{-\infty}^\infty$ of closed subspaces of $\mathcal H$
is called a {\emph generalized multiresolution analysis} (GMRA) relative to $\pi$ and $\delta$ if
\begin{enumerate}
\item\hskip2em $V_j\subseteq V_{j+1}$ for all $j.$
\item\hskip2em $V_{j+1}=\delta(V_j)$ for all $j.$
\item\hskip2em $\cap V_j=\{0\},$ and $\cup V_j$ is dense in $\mathcal H.$
\item\hskip2em $V_0$ is invariant under the representation $\pi.$
\end{enumerate}
The subspace $V_0$ is called the {\bf core subspace} of the GMRA $\{V_j\}.$
\end{definition}
In order to use the theorems from the previous section to build GMRA's from filters, we need to know that associated Ruelle operators are isometries.  The proof requires the additional assumption that the multiplicity function $m$ is finite a.e.  This hypothesis is standard in much of the literature.    
\begin{prop}
Assume $m(\omega) < \infty$ for almost all $\omega,$
and let $H$ be a filter relative to $m$ and $\alpha^*.$
Then
the Ruelle operator $S_H$ is an isometry of $\bigoplus_i L^2(\sigma_i,\mu)$ into itself.
\end{prop}
\begin{proof}
Note that, because $m(\omega)<\infty$ almost everywhere,
the filter equation, together with the very definition of $h_{i,j},$ implies that $h_{i,j}(\omega) = 0$
if $j>m(\omega)$ or $i>m(\alpha^*(\omega)).$
Therefore, all the sums in the following calculation, that are inside integrals, are finite, so that
interchanges of these sums is allowed.
\begin{align*}
\|S_H(f)\|^2 & = \sum_j \int_{\sigma_j} |[S_H(f)]_j(\omega)|^2 \, d\omega \cr
& = \sum_j\int_{\sigma_j} |[H^t(\omega)f(\alpha^*(\omega))]_j|^2 \, d\omega \cr
& = \sum_j\int_{\sigma_j} \left|\sum_i H_{i,j}(\omega)f_i(\alpha^*(\omega))\right|^2 \, d\omega \cr
& = \sum_j\int_{\widehat\Gamma} \left|\sum_i h_{i,j}(\omega) f_i(\alpha^*(\omega))\right|^2\, d\omega \cr
& = \int_{\widehat\Gamma} \sum_j \left[\sum_i h_{i,j}(\omega) f_i(\alpha^*(\omega))\right]
\overline{\left[\sum_{i'} h_{i',j}(\omega) f_{i'}(\alpha^*(\omega))\right]} \, d\omega \cr
& = \frac1N \sum_{\alpha^*(\zeta)=1}  \int_{\widehat\Gamma} \sum_j \left[\sum_i h_{i,j}(\omega\zeta) f_i(\alpha^*(\omega))\right]
\overline{\left[\sum_{i'} h_{i',j}(\omega\zeta) f_{i'}(\alpha^*(\omega))\right]} \, d\omega \cr
& = \frac1N \int_{\widehat\Gamma} \sum_i \sum_{i'} \sum_\zeta \sum_j
h_{i,j}(\omega\zeta)\overline{h_{i',j}(\omega\zeta)} f_i(\alpha^*(\omega)) \overline{f_{i'}(\alpha^*(\omega))} \, d\omega \cr
& = \int_{\widehat\Gamma} \sum_i \chi_{\sigma_i}(\alpha^*(\omega)) |f_i(\alpha^*(\omega))|^2\, d\omega \cr
& = \sum_i \int_{\sigma_i} |f_i(\omega)|^2\, \cr
& = \|f\|^2,
\end{align*}
as claimed.
\end{proof}

\par
In earlier works (e.g., \cite{BJMP}, and \cite{IJKLN}),
a so-called ``low-pass condition'' on the filter $H$ was used
to guarantee that $S_H$ was a pure isometry.
This condition had various forms, but they all required something like $H$ being
continuous at the identity 1 in $\widehat\Gamma$
and the matrix $H(1)$ being diagonal
with $\sqrt N$'s at the top of the diagonal and 0's at the bottom.  Results from \cite{PSWX} and more recently \cite{BR} loosened these assumptions somewhat by separating out a phase factor.  We will show below that such assumptions on $H$ imply that $S_H$ can have no eigenvector, and so by Theorem \ref{eigenvalue}, $S_H$ must be a pure isometry.
 The following theorem gives quite general conditions on $H$ under which $S_H$ is a pure isometry, subsuming the conditions on $H(1)$ mentioned above as well as the results for a $1\times 1$ filter $H$ given in \cite{AIJLN}.  In particular, note that this theorem blurs the distinction between classical low-pass and high-pass filters by not requiring $H$ to take on specific values near the identity.

\begin{theorem}
\label{low-pass}
Let $H$ be a filter relative to $m$ and $\alpha^*$.  
Suppose there exists a positive number $\delta $ and a set $F\subseteq\widehat\Gamma$ of positive measure,
such that for all $\omega\in F$ the matrix $H(\omega)$ is in block form
\[
H(\omega) = \left(\begin{matrix} A(\omega) & B(\omega)\\
C(\omega) & D(\omega) \end{matrix}\right),
\]
where the four blocks satisfy the following:
\begin{enumerate}
\item\hskip2em $A(\omega)$ is a square expansive matrix with the property 
that $\|A(\omega)^{-1}\|\leq\frac1{1+\delta}.$
\item\hskip2em $\max(\|B(\omega)\|,\|C(\omega)\|,\|D(\omega)\|) < \epsilon =\min(\frac18,\frac\delta8).
$
(The norm here can either be the operator norm of a matrix
or the Euclidean norm.)
\end{enumerate}
Finally, assume that $F\cap \alpha^*(F)$ also has positive measure.
Then $S_H$ is a pure isometry, i.e.,
$S_H$ has no eigenvector.
\end{theorem}
\begin{remark}
The hypothesis of this theorem
clearly covers the previously cited cases where $H(\omega)$ is continuous and has the relevant diagonal at $\omega=1.$
\end{remark}

\begin{proof}
Suppose, by way of contradiction, that $f$ is a unit eigenvector for $S_H$ with eigenvalue $\lambda$, $|\lambda|=1$.
For each $\omega,$ write the vector $f(\omega)$
in the form $f(\omega) = (f^1(\omega),f^2(\omega)),$
where $f^1(\omega)$ is $a$-dimensional.
Because $f$ is an eigenvector for $S_H,$ we have 
\[
\lambda f(\omega) = [S_H(f)](\omega) = H^t(\omega) f(\alpha^*(\omega)).
\]
It follows from this, and the fact that $\|f(\omega)\|=1$ by the final conclusion of Theorem \ref{eigenvalue}, that for $\omega\in F$ we must have
\[
\|f^2(\omega)\| = \|B^t(\omega)f^1(\alpha^*(\omega)) + D^t(\omega)f^2(\alpha^*(\omega))\| < 2\epsilon.
\]
Hence, again because $\|f(\omega)\|=1$ for $\omega\in F$, we must have
$\|f^1(\omega)\| > 1-2\epsilon.
$
Since condition (1) on the matrix $A(\omega)$ implies that
$\|A(\omega) v\| \geq (1+\delta)\|v\|$ for
every $a$-dimensional vector $v$, we must have, for $\omega$ and $\alpha^*(\omega)$ both in $F,$
\begin{align*}
1 & \geq \|f^1(\omega)\| \cr
& = \|A^t(\omega)f^1(\alpha^*(\omega)) + C^t(\omega)f^2(\alpha^*(\omega))\| \cr
& > (1+\delta) \|f^1(\alpha^*(\omega))\| - 2\epsilon^2 \cr
& > (1+\delta)(1-2\epsilon) -2\epsilon \cr
& \geq 1+\delta -\frac\delta4 -\frac\delta4 - \frac\delta4 \cr
& = 1+\frac\delta4.
\end{align*}
We have arrived at a contradiction, and the theorem is proved.
\end{proof}

Theorem \ref{low-pass} can be used to build generalized multiresolution analyses with more general filters, using approaches that do not require an infinite product construction, such as the direct limit construction in \cite{IJKLN} and \cite{AIJLN}.  

\begin{example}
For another application of Theorem \ref{low-pass}, consider the Journ\'e filter system given by 
$$H(\omega)\;=\;\left(\begin{array}{rr}
h_{1,1}&h_{1,2}\\
h_{2,1}&h_{2,2}
\end{array}\right).$$
In the classical Journ\'e example described by Baggett, Courter, and Merrill in \cite{BCM}, 
$$h_{1,1}\;=\;\sqrt{2}e^{2\pi i\chi_{E_1}},\;h_{1,2}=0,$$
$$h_{2,1}\;=\;\sqrt{2}e^{2\pi i \chi_{E_2}},\;h_{2,2}\;=0,$$
where 
$$E_1=\left[-\frac{2}{7},-\frac{1}{4}\right)\cup \left[-\frac{1}{7},\frac{1}{7}\right)\cup \left[\frac{1}{4},\frac{2}{7}\right),$$
$$E_2= \left[-\frac{1}{2},-\frac{3}{7}\right)\cup \left[\frac{3}{7},\frac{1}{2}\right).$$

These are the classical filters in $L^2(\mathbb R)$ for dilation by $2$ that can be associated to the GMRA in $L^2(\mathbb R)$ coming from the Journ\'e wavelet.  

We now apply a device very similar to that first used on pp. 259--260 of \cite{BJMP2}.
Choose a very small $\delta>0,$ and an even smaller $\varepsilon>0.$  Define $q:\mathbb T\to \mathbb R$ by

$$q(e^{2\pi ix})\;=\left\{\begin{array}{rrrrrrrrr}
{\sqrt{2}\sqrt{1-r^2},}&\mbox{if}\;x\;=0,\\
{0,}&\mbox{if }\;\frac{1}{7}-\varepsilon<x<\frac{3}{14}+\varepsilon,\\
{C^{\infty}\;\mbox{monotone decreasing},}&\mbox{if}\;0<x\;<\frac{1}{7}-\varepsilon,\\
{\sqrt{2},}&\mbox{if}\;\frac{2}{7}-\varepsilon<x<\frac{5}{14}+\varepsilon,\\
{C^{\infty}\;\mbox{monotone increasing},}&\mbox{if}\;\frac{3}{14}+\varepsilon<x\;<\frac{2}{7}-\varepsilon,\\
{0,}&\mbox{if}\;\frac{3}{7}-\varepsilon<x<\frac{3}{7}+\varepsilon,\\
{\sqrt{2}\cdot r,}&\mbox{if}\;x\;=\frac{1}{2},\\
{C^{\infty}\;\mbox{monotone increasing},}&\mbox{if}\;\frac{3}{7}+\varepsilon<x\;<\frac{1}{2},\\
{\sqrt{2-[q(e^{2\pi i(x+\frac{1}{2})})]^2},} &\mbox{if}\;-\frac{1}{2}\;<x\;<\;0.
\end{array}\right.$$
Here $r\in (0,1)$ is a number yet to be determined.
We now define generalized filters $(h^q_{i,j})$ by:
$$h^q_{1,1}(e^{2\pi ix})\;=q(e^{2\pi ix})e^{2\pi i\chi_{[-\frac{2}{7},\frac{2}{7})}(x)},$$
$$h^q_{2,1}(e^{2\pi ix})\;=\;\sqrt{2}e^{2\pi i\chi_{[-\frac{1}{2},-\frac{3}{7})\cup [\frac{3}{7},\frac{1}{2})}(x)},$$
$$h^q_{1,2}(e^{2\pi i x})\;=\;q(e^{2\pi i(x+\frac{1}{2})})e^{2\pi i\chi_{[-\frac{1}{7},\frac{1}{7})}(x)},$$
$$h^q_{2,2}(e^{2\pi ix})\;=\;0.$$
Denote the matrix $(h^q_{i,j}(z))$ by $H^q.$  A routine calculation shows that the filter equations of Baggett, Courter and Merrill are satisfied, i.e. 
$$\sum_{j=1}^2\sum_{k=0}^1h_{i,j}\left(e^{2\pi i\frac{x+k}{2}}\right)\overline{h_{i',j}\left(e^{2\pi i\frac{x+k}{2}}\right)}\;=\;\delta_{i,i'}2\chi_{_{\sigma_i}}(x),\;i=1,\;2,$$
where $\sigma_1=[-\frac{1}{2},-\frac{3}{7})\cup [-\frac{2}{7},\frac{2}{7})\cup [-\frac{3}{7}, \frac{1}{2}),$ and $\sigma_2=[-\frac{1}{7},\frac{1}{7}).$ 

We now take $A(z)=h^q_{1,1}(z),\;B(z)=h^q_{1,2}(z),\;C(z)=h^q_{2,1}(z),$ and
$D(z)=h^q_{2,2}(z)$ in the matrix $H=H^q.$ We want to determine a specific value $r\in (0,1)$ and   a set $F=\{e^{2\pi i x}: x\in\mathcal F\}$, where $\mathcal F\subset [-\frac{1}{2},\frac{1}{2})$ such that the hypotheses of Theorem \ref{low-pass} are satisfied.  We let $\mathcal F=[-\frac{1}{n},\frac{1}{n}],$ where $n\in\mathbb N$ is chosen so that $n\geq 7$ and $q(e^{2\pi ix})>\sqrt{2}\sqrt{1-2r^2},$ for all $x\in \mathcal F.$  This can be done by applications of the Intermediate Value Theorem, since $q$ is continuous.
It's clear that $F\cap \alpha^*(F)=F$ has positive measure.  Also we want to find $\delta>0$ such that 
$|h^q_{1,1}(e^{2\pi ix})|\geq 1+\delta$ and 
$\max(|h^q_{1,2}(e^{2\pi ix})|,|h^q_{2,1}(e^{2\pi ix})|, |h^q_{2,2}(e^{2\pi i x})|) < \epsilon =\min(\frac18,\frac\delta8),\;\forall x \in \mathcal F.$
Note that $h^q_{2,1}(e^{2\pi ix})$ and $h^q_{2,2}(e^{2\pi i x})$ are identically $0$ on $\mathcal F,$ so we need only show that $|h^q_{1,2}(e^{2\pi i x})|<\epsilon$ on $\mathcal F.$  Since $h^q_{1,2}(e^{2\pi i x})$ is continuous at $x=0,$ where its value is equal to $\sqrt{2}\cdot r,$ we choose $\mathcal F=[-\frac{1}{n},\frac{1}{n}]$ so that $\sqrt{2}\cdot r\leq h^q_{1,2}(e^{2\pi i x})<2\cdot r,\;\forall x\in \mathcal F.$ 

Having chosen $\delta>0,$ we thus must choose $r$ so that $\frac{1}{\sqrt{2}\sqrt{1-2r^2}}\leq \frac{1}{1+\delta}$ and $2r<\epsilon =\min(\frac18,\frac\delta8).$  So we first choose $r_1<\min(\frac{1}{16},\frac{\delta}{16}).$

For $r_2,$ as long as $\delta<\sqrt{2}-1,$ if we choose $r_2\leq \sqrt{\frac{\sqrt{2}-(1+\delta)}{1+\delta}},$ one can verify that 
$$\frac{1}{\sqrt{2}\sqrt{1-2r_2^2}}\leq \frac{1}{1+\delta}.$$

Finally, we choose $r=\min(r_1,r_2).$  
Then, $$2r<\min\left(\frac18,\frac\delta8\right),$$
and $$\frac{1}{\sqrt{2}\sqrt{1-2r^2}}\leq \frac{1}{1+\delta},$$ so that the conditions of Theorem \ref{low-pass} are satisfied, and $S_H$ is a pure isometry acting on $L^2(\sigma_1)\oplus L^2(\sigma_2).$  In fact, letting $r\to 0+,$ we can construct a one-parameter family of filter systems giving rise to pure isometries; when $r=0,$ we obtain exactly the filter system constructed in \cite{BJMP2}.

In Theorem 5 of \cite{IJKLN}, it is shown that given a pure isometry $S$ on a Hilbert space ${\mathcal K}$ together with a representation $\rho$ of a countable abelian group $\Gamma,$ such that $\delta^{-1}\rho_{\gamma}\delta=\rho_{\alpha(\gamma)}$ for all $\gamma\in \Gamma,$ then it was possible to construct a generalized multiresolution analysis via a direct limit process.  Taking $S=S_H,$ and $\Gamma=\mathbb Z,$ the desired hypotheses will be satisfied, and it follows that a GMRA can be constructed from the above filter system.  In a paper in preparation, the authors will present a more constructive approach to making the GMRA under the same hypotheses as in Theorem 5 of \cite{IJKLN}.

\end{example}

\section{Pure isometries and the trivial intersection property}

The following ``problem'' was first noticed by Baggett, Bownik and Rzeszotnik.
Suppose $\{\psi_k\}$ is a Parseval multiwavelet in $L^2(\mathbb R^d);$
i.e., the functions $\{\psi_{j,n,k}(x)\}\equiv \{\sqrt2^j \psi_k(2^jx+n)\}$ form a Parseval frame for all of
$L^2(\mathbb R^d).$
If $V_j$ is defined to be the closed linear span of the functions
$\{\psi_{l,n,k}\}$ for $l<j,$ then
these subspaces can be shown to satisfy all of the properties of a GMRA
except for the condition $\cap V_j=\{0\}.$
Bownik and Rzeszotnik demonstrated the delicacy of this condition in \cite{BR05} by constructing, for any $\delta>0$, a frame wavelet in $L^2(\R)$, with frame bounds of $1$ and $1+\delta$, that has a negative dilate space $V_0$ equal to all of $L^2(\R)$.  They showed in \cite{BR}, however, that a Parseval multiwavelet in $L^2(\R^d)$ generates a GMRA (that is, the trivial intersection property does hold) whenever the multiplicity function of the negative dilate space $V_0$ is finite on a set of positive measure.  In fact, Bownik proved in \cite{Bow} that the condition that $m$ is not identically $\infty$ a.e. implies $\cap_{j=1}^{\infty} D_j(V_0) = \{0\}$ in the more general setting where $D_jf(x) = f(A_jx)$ for a sequence $\{A_j\}$ of invertible $n\times n$ real matrices that satisfy $\|A_j\|\to 0$ as $j\to\infty$.  For a history of the intersection problem in $L^2(\R^d)$, see \cite{Bow08}.

This question about subspaces of $L^2(\mathbb R^d)$  obviously generalizes to a collection
$\{V_j\}$ of subspaces  of a Hilbert space that satisfy all the conditions for a GMRA
except the trivial intersection condition.
Below, we apply the results from Section \ref{pure} to show that this intersection is $\{0\}$
if certain extra assumptions hold.  In doing so, we extend some of the results mentioned the previous paragraph.

\par
Let $\Gamma$, $\alpha$, $\pi$, and $\delta$ be as in the previous sections.  We recall some implications of Stone's Theorem, whereby certain GMRAs give rise to an associated filter.  Let $\{V_j\}$ be a GMRA in a Hilbert space $\mathcal H,$ relative to the representation $\pi$ and the operator $\delta.$  Then, according to Stone's Theorem on unitary representations of abelian groups, there exists a finite, Borel measure $\mu$ (unique up to equivalence of measures) on $\widehat\Gamma,$ unique (up to sets of $\mu$ measure 0) Borel subsets $\sigma_1\supseteq \sigma_2 \supseteq \ldots$ of $\widehat\Gamma,$ and a (not necessarily unique) unitary operator 
$J:V_0 \to \bigoplus_i L^2(\sigma_i,\mu)$ satisfying
\[
[J(\pi_\gamma(f))](\omega) = \omega(\gamma) [J(f)](\omega)
\]
for all $\gamma\in\Gamma,$ all $f\in V_0,$ and $\mu$ almost all $\omega\in\widehat\Gamma.$
In this paper, we assume the measure $\mu$ is absolutely continuous with respect to Haar measure,
in which case we may assume that $\mu$ is the restriction of Haar measure to the subset $\sigma_1.$  

Write $C_i$ for the element of the direct sum space
$\bigoplus_j L^2(\sigma_j,\mu)$ whose $i$th coordinate is $\chi_{\sigma_i}$
and whose other coordinates are 0.
Write $\bigoplus_j h_{i,j}$ for the element $J(\delta^{-1}(J^{-1}(C_i))).$
The following theorem displays a connection between
GMRA structures and filters and will allow us to apply Proposition \ref{superfilter} and Theorem \ref{eigenvalue}.

\begin{theorem}\label{lowpassfilter}
Let the functions $\{h_{i,j}\}$ be as in the preceding paragraph.
Then the matrix $H=\{h_{i,j}\}$ is a filter relative to $m$ and $\alpha^*.$
Moreover, the operator $J\circ \delta^{-1}  \circ J^{-1}$
on $\bigoplus_i L^2(\sigma_i,\mu)$
is the corresponding Ruelle operator $S_H:$
\[
[J\circ \delta^{-1}  \circ J^{-1}(f)](\omega) = H^t(\omega) f(\alpha^*(\omega)).
\]
\end{theorem}
\begin{proof}
By definition we have
\begin{align*}
\mu(\sigma_i) & = \|C_i\|^2 \cr
& = \|J(\delta^{-1}(J^{-1}(C_i)))\|^2 \cr
& = \sum_j \int_{\widehat\Gamma} |h_{i,j}(\omega)|^2\, d\omega,
\end{align*}
which implies that $\sum_j |h_{i,j}(\omega)|^2$ is finite
for almost all $\omega.$
Write
\[
F_{i,i'}(\omega) = \sum_{\alpha^*(\zeta) = 1} \sum_j h_{i,j}(\omega\zeta) \overline{h_{i',j}(\omega\zeta)},
\]
and note, by the Cauchy-Schwarz Inequality, that $F_{i,i'}\in L^1(\mu),$
and that the Fourier coefficient $c_\gamma(F_{i,i'}) = 0$ unless $\gamma$ belongs to the range of $\alpha.$
We have that
\begin{align*}
c_{\alpha(\gamma)}(F_{i,i'}) & = \int_{\widehat\Gamma} F_{i,i'}(\omega) \omega(-\alpha(\gamma))\, d\mu(\omega) \cr
& =  \sum_\zeta \int_{\widehat\Gamma} \sum_j h_{i,j}(\omega\zeta) \overline{h_{i',j}(\omega\zeta)} \omega(-\alpha(\gamma))\, d\omega \cr
& = N\sum_j  \int_{\widehat\Gamma} h_{i,j}(\omega) \overline{h_{i',j}(\omega)} \omega(-\alpha(\gamma))\, d\omega \cr
& = N\langle J(\delta^{-1}(J^{-1}(C_i))) \mid J(\pi_{\alpha(\gamma)} (\delta^{-1}(J^{-1}(C_{i'}))))\rangle \cr
& = N\langle J^{-1}(C_i) \mid \pi_\gamma (J^{-1}(C_{i'}))\rangle \cr
& = N\langle C_i \mid \gamma C_{i'} \rangle \cr
& = N \delta_{i,i'} \langle C_i \mid \gamma C_{i'}\rangle \cr
& = N \delta_{i,i'} \int_{\widehat\Gamma} \chi_{\sigma_i}(\omega) \omega(-\gamma)\, d\omega \cr
& = N \delta_{i,i'} \int_{\widehat\Gamma} \chi_{\sigma_i}(\alpha^*(\omega)) \alpha^*(\omega)(-\gamma)\, d\omega,
\end{align*}
showing that the two $L^1$ functions $F_{i,i'}(\omega)$ and $N\delta_{i,i'}\chi_{\sigma_i}(\alpha^*(\omega))$
have the same Fourier coefficients, and hence are equal almost everywhere.
This verifies Equation (\ref{filterequation}).
It follows from the filter equation that
$h_{i,j}$ is supported on
${\alpha^*}^{-1}(\sigma_i).$
That is, $h_{i,j}(\omega) = 0$ unless both $\omega\in \sigma_j$ and $\alpha^*(\omega) \in \sigma_i.$
\par
Next, for any $\gamma,$ we have
\begin{align*}
[J(\delta^{-1}(J^{-1}(\gamma C_i)))](\omega)
& = [J(\delta^{-1}(\pi_\gamma(J^{-1}(C_i))))](\omega) \cr
& = [J(\pi_{\alpha(\gamma)}(\delta^{-1}(J^{-1}(C_i))))](\omega) \cr
& = \omega(\alpha(\gamma)) \bigoplus_j h_{i,j}(\omega) \cr
& = \omega(\alpha(\gamma)) \chi_{\sigma_i}(\alpha^*(\omega)) \bigoplus_j h_{i,j}(\omega) \cr
& = H^t(\omega) [\gamma C_i](\alpha^*(\omega)).
\end{align*}
Then, by the Stone-Weierstrass Theorem, we must have
\[
[J(\delta^{-1}(J^{-1}(fC_i)))](\omega) = H^t(\omega) [fC_i](\alpha^*(\omega))
\]
for every continuous function $f$ on $\widehat\Gamma.$
Then, by standard integration methods, this
equality holds for all $L^2$ functions $f.$
Finally, if $F=\bigoplus_i f_i,$ then
\begin{align*}
[J(\delta^{-1}(J^{-1}(F)))](\omega) 
& = \left[J\left(\delta^{-1}\left(J^{-1}\left(\sum_i f_iC_i\right)\right)\right)\right](\omega) \cr
& = \sum_i [J(\delta^{-1}(J^{-1}(f_iC_i)))](\omega) \cr
& = \sum_i H^t(\omega) f_i(\alpha^*(\omega) C_i(\alpha^*(\omega)) \cr
& = H^t(\omega) F(\alpha^*(\omega)) \cr
& = [S_H(F)](\omega),
\end{align*}
proving the second assertion.
\end{proof}
We call this matrix function $H$ a {\bf low-pass filter determined by the GMRA}.
\begin{remark}
The preceding proof works in a more general setting.
That is, we do not use all of the GMRA structure, particularly the
property that $\cap V_j = \{0\}.$  In particular, $\cap V_j=\{0\}$ if and only if $S_H = J \circ \delta^{-1} \circ J^{-1}$ is a pure isometry. 
\end{remark}

We introduce two more groups.
Let $D$ be the direct limit group determined by
$\Gamma$ and the monomorphism $\alpha$  of $\Gamma$ into itself.  (See for example \cite{IN}.)  For clarity, we make this construction explicit as follows.  

Let $\widetilde\Gamma$ be the set of all pairs $(\gamma,j)$ for $\gamma\in\Gamma$ and $j$ a nonnegative integer.
Define an equivalence relation on $\widetilde\Gamma$ by $(\gamma,k) \equiv (\gamma',k')$ if and only if
$\alpha^{k'}(\gamma) = \alpha^k(\gamma'),$ and let $D$ be the set of equivalence classes $[\gamma,k]$ of this relation.
Define addition on $D$ by
\[
[\gamma_1,k_1] + [\gamma_2,k_2] = [\alpha^{k_2}(\gamma_1) + \alpha^{k_1}(\gamma_2),k_1+k_2].
\]
One verifies directly that this addition is well-defined,
and that $D$ is an abelian group.
\par
Next, define a map $\widetilde\alpha$ on $D$ by $\widetilde\alpha([\gamma,k]) = [\alpha(\gamma),k].$
Again, one verifies directly that $\widetilde\alpha$ is well-defined and that it is an
isomorphism of $D$ onto itself.  Indeed, the inverse $\widetilde\alpha^{-1}$
is given by $\widetilde\alpha^{-1}([\gamma,k]) = [\gamma,k+1].$
\par
Define $G$ to be the semidirect product $D\rtimes \mathbb Z,$
where the integer $j$ acts on the element $d$ by $j\cdot d=\widetilde\alpha^j(d).$
Explicitly, the multiplication in $G$ is given by
\[
(d_1,j_1)\times (d_2,j_2) = (d_1+\widetilde\alpha^{j_1}(d_2),j_1+j_2).
\]
\par
As before, $\pi$ is a unitary representation of $\Gamma,$ acting in a Hilbert space $\mathcal H,$
and $\delta$ a unitary operator on $\mathcal H$ for which
\[
\delta^{-1}\pi_\gamma \delta = \pi_{\alpha(\gamma)}
\]
for all $\gamma\in \Gamma.$
Define a representation $\widetilde\pi$ on $G$ by
\[
\widetilde\pi_{(d,j)} = \widetilde\pi_{([\gamma,k],j)} = \delta^k\pi_\gamma \delta^{-k-j}.
\]
One verifies directly that this is a representation of $G.$  Note also that
$\widetilde\pi_{\widetilde\alpha(d)} = \delta^{-1} \widetilde\pi_d \delta$ and
$\widetilde\pi_{(d,j)} = \widetilde\pi_{d} \delta^{-j}.$

\par
Finally, for $|\lambda|=1$, define the (irreducible) unitary representation $P^{\lambda}$ of $G$ acting in the Hilbert space $l^2(D)$ by 
\[
[P^{\lambda}_{(d,j)}(f)](d') = \lambda^jf({\widetilde\alpha}^{-j}(d'-d)).
\]
\begin{remark}
The representation $P^{\lambda}$ is equivalent to the induced representation
$\mbox{Ind}^G_{\mathbb Z}\chi_{_{\lambda}},$ where $\chi_{_{\lambda}}$ is the character of the subgroup  ${\mathbb Z}$ determined by $\lambda$.

\end{remark}
\begin{theorem}\label{intersection}
Suppose $\{V_j\}$ is a collection of
closed subspaces of $\mathcal H$ that satisfy all the conditions for a GMRA, relative to $\pi$ and $\delta$,
except possibly the condition that $\cap V_j=\{0\}.
$
Assume that the measure $\mu$ associated to
the representation $\pi$ restricted to $V_0$ is Haar measure, and that the multiplicity function $m$ is finite on a set of positive measure.
Then the following conditions are equivalent:
\begin{enumerate}
\item $\cap V_j\neq \{0\}.
$
\item $\delta$ has an eigenvector.
\item The representation $\widetilde\pi$ of $G$ contains a subrepresentation
equivalent to the representation $P^{\lambda}$ for some $|\lambda|=1$.
\end{enumerate}
\end{theorem}
\begin{proof}
We define the functions $\{h_{i,j}\}$ as in Theorem \ref{lowpassfilter} and reiterate that the matrix-valued function $H$ is a filter relative to the space $V_0$ and that the operator $J\circ \delta^{-1}\circ J^{-1}$ is the Ruelle operator.
  Then $S_H$ is a composition of isometries, and thus clearly an isometry from $\bigoplus_j L^2(\sigma_j,\mu)$ into itself.
  \par
  Assume (1), and thus that $S_H$ is not pure.  
  From Theorem \ref{eigenvalue}, we know that
$S_H$ has a unit eigenvector $f$:  
$S_H(f) = \lambda f,
$
and $|\lambda|=1.
$
(We are using here the hypothesis that $m(\omega)<\infty$ on a set of positive measure.)
For such an eigenfunction $f,$
$v=J^{-1}(f)$ is an eigenvector for $\delta^{-1}.$ 
This proves (1) implies (2).
\par
Assume (2), and let $v$ be a unit eigenvector for $\delta$ with eigenvalue $\lambda$.
Because $\pi$ is equivalent to a subrepresentation of some multiple
of the regular representation of $\Gamma,$ we must have, from the Riemann-Lebesgue Lemma, that the function
$\langle\pi_\gamma(w)\mid w\rangle$ vanishes at infinity on $\Gamma$ for every $w\in\mathcal H.
$
But, for each $\gamma,$ we have
\begin{align*}
|\langle \pi_{\alpha^j(\gamma)}(v) \mid v\rangle| 
& = |\langle \delta^{-j}\pi_\gamma \delta^j(v)\mid v \rangle| \cr
& = |\langle \pi_\gamma(v)\mid v\rangle|,
\end{align*}
implying then that $\langle \pi_\gamma(v) \mid v\rangle = 0$ for all $\gamma\neq 0,$
or equivalently that $\langle\pi_\gamma(v)\mid \pi_{\gamma'}(v)\rangle = 0$
unless $\gamma = \gamma'.
$
But now, for $d=[\gamma,k]\in D,$ we have
\begin{align*}
\langle \widetilde\pi_d(v)\mid v\rangle
& = \langle \widetilde\pi_{[\gamma,k]}(v) \mid v\rangle \cr
& = \langle \delta^k\pi_\gamma \delta^{-k}(v) \mid v\rangle \cr
& = \langle \pi_\gamma(v)\mid v\rangle \cr
& = 0
\end{align*}
unless $\gamma= 0.$  It follows then that
the vectors $\{\widetilde\pi_d(v)\}$ form an orthonormal set.
The span $X$ of these vectors is obviously an invariant subspace
for the representation $\widetilde\pi$ of $G.
$  Moreover,
we claim that the restriction of $\widetilde\pi$ to $X$ is equivalent to the representation $P^{\lambda}.
$
Thus, let $U$ be the unitary operator from $X$ onto $l^2(D)$
that sends the basis vector $\widetilde\pi_d(v)$ to the point mass basis vector $\epsilon_d$
in $l^2(D).$
We have
\begin{align*}
U(\widetilde\pi_{(d,j)}(\widetilde\pi_{d'}(v)))
& = U(\widetilde\pi_d \delta^{-j} \widetilde\pi_{d'}(v)) \cr
& = U(\widetilde\pi_d \widetilde\pi_{\widetilde\alpha^j(d')} \delta^ {-j}(v)) \cr
& =\lambda^{-j} U(\widetilde\pi_{d+\widetilde\alpha^j(d')} (v)) \cr
& = \lambda^{-j}\epsilon_{d+\widetilde\alpha^j(d')} \cr
& = P^{\overline\lambda}_{(d,j)}(\epsilon_{d'}) \cr
& = P^{\overline\lambda}_{(d,j)} (U(\widetilde\pi_{d'}(v))),
\end{align*}
showing the equivalence of the restriction of $\widetilde\pi$ to the subspace $X$
and the representation $P^{\overline\lambda}.$
This proves (2) implies (3).
\par
Assume (3).  Because the operator $\widetilde\pi_{(0,1)}$ is $\delta^{-1},$ it follows that
there is a nonzero vector $v\in \mathcal H$ for which $\delta^{-1}(v) = \widetilde\pi_{(0,1)}(v) = \lambda v.
$
(Just note that $P^{\lambda}_{(0,1)}$ has an eigenvector with eigenvalue $\lambda$.
This shows (3) implies (2).

Finally, assume (2), and let $v$ be an eigenvector for $\delta$ with eigenvalue $\lambda$.
Write $W_j$ for the orthogonal complement of $V_j$ in $V_{j+1}.$  Then
\[
\mathcal H = \bigoplus_{j=-\infty}^\infty W_j \oplus R_\infty,
\]
where $R_\infty = \cap V_j,
$
so we may write
\[
v = \sum_{j=-\infty}^\infty  v_j + v_\infty,
\]
where $v_j$ is the projection of $v$ onto the subspace $W_j,$
and $v_\infty$ is the projection of $v$ onto the subspace $R_\infty.
$
Applying the operator $\delta$ gives
\[
\sum_j  v_j +  v_\infty =  v =\overline\lambda \sum_j\delta(v_j) + \overline\lambda \delta(v_\infty),
\]
implying that $ v_{j+1}=\overline\lambda\delta(v_j),$ whence
$\|v_{j+1}\| = \|v_j\|$ for all $j.
$
Therefore $v_j=0$ for all $j,$ and hence
$v=v_\infty.
$  So, $R_\infty \neq \{0\},$ and (2) implies (1).
(This part of the proof uses neither the hypothesis on $m$
nor the one on $\pi.
$)
This completes the proof of the theorem.
\end{proof}
\begin{remark}
We note that the proofs of Theorems \ref{eigenvalue} and \ref{intersection} imply that $\delta$ has an eigenvector whenever $\cap V_j$ is nontrivial but finite-dimensional.  Therefore, $\cap V_j$ must be infinite-dimensional whenever it is nontrivial.
\end{remark}
\begin{example}
Let $\mathcal H = L^2(\mathbb R),$ let $\Gamma=\mathbb Z$
and $\pi$ be the representation of $\Gamma$ determined by translation.
Let $\alpha(k) = 2k,$ and define $\delta$ on $\mathcal H$ by
$[\delta(f)](x) = \sqrt2 f(2x).
$
Then $\delta^{-1} \pi_k \delta = \pi_{2k}.
$
Set $V_j$ equal to the subspace of $\mathcal H$ comprising those functions $f$
whose Fourier transform is supported in the interval $(-\infty,2^j).
$
Then the $V_j$'s satisfy all the conditions for a GMRA relative to $\pi$ and $\delta$ except the trivial intersection condition.
Indeed, the intersection is nontrivial, because $\cap V_j$ is the subspace of functions whose
Fourier transform is supported in the interval $(-\infty,0].
$
The subspace $V_0$ comprises the functions whose transform is
supported in the interval $(-\infty,1),$ and it follows that the multiplicity function
associated to the restriction of $\pi$ to this subspace is infinite everywhere.
Hence, since $\delta$ has no eigenvector, we see that we cannot drop the hypothesis that $m<\infty$ on a set of positive measure from the theorem above.
\end{example}
\begin{example}
Now let $\Gamma= \mathbb Z^2,$ let $\alpha(n,k) = (2n,2k),$ let
$\mathcal H = L^2(\mathbb R^2),$
let $\pi$ be the representation of $\Gamma$ on $\mathcal H$ determined by translation,
and let $[\delta(f)](x,y) = 2f(2x,2y).
$
Then, $\delta^{-1}\pi_\gamma \delta = \pi_{\alpha(\gamma)}.
$
Let $V_j$ be the subspace of $\mathcal H$ comprising the
functions whose Fourier transform is supported in the rectangle
$(-\infty,2^j)\times (-2^j,2^j).
$
Then, the multiplicity function associated to the restriction of
$\pi$ to $V_0$ is infinite everywhere, but
this time $\cap  V_j = \{0\}.
$
Indeed, if $f\in \cap V_j,$ then the support of the
Fourier transform of $f$ is supported on the negative $x$-axis, and such a function
is the 0 vector in $L^2(\mathbb R^2).
$
Hence, the finiteness assumption of Theorem \ref{intersection} on $m$ is not necessary
for the intersection to be trivial.
\end{example}
The following example shows that for a Hilbert space $\mathcal H\neq L^2(\mathbb R^d)$, a collection of subspaces $\{V_j\}$ can satisfy all of the properties of a GMRA except the trivial intersection property, even though the multiplicity function is finite almost everywhere.
  As Theorem \ref{intersection} shows, this is made possible by the presence of an eigenvector for the dilation $\delta$.
  The implication (3) implies (2) in Theorem \ref{intersection} suggests a method of construction of such examples.

\begin{example}
\label{fixed vector}
Let $\mathcal H=l^2(D)$, where $D$ is the group of dyadic rationals.
  The group $\Gamma=\mathbb Z$ acts on this space by  $\pi_{\gamma}(f)(x)=f(x-\gamma)$, and if we define $\alpha(\gamma)=2\gamma$ and $\delta f(x)=f(2x)$, we see that $\delta^{-1}\pi_\gamma \delta = \pi_{\alpha(\gamma)}$.
  Note that $f=\chi_{0}$ is a fixed vector for this dilation.  Now take $V_0$ to be the subspace $l^2(\mathbb Z)$, and $V_j=\delta^j V_0$.
  Since $\pi$ is the regular representation of $\Gamma$ on $V_0$, the multiplicity function $m$ is identically 1.
  We see that the fixed vector $\chi_0$ is a nonzero element of the intersection of the $V_j$.
\end{example}
Finally, to clarify how Theorem \ref{intersection} fits in with previous results about the intersection problem, let $\mathcal H,$ $\pi,$ and $\delta$ be as in the beginning of this section.
If $\{\psi_i\}$ is a set of vectors in $\mathcal H$ for which the collection
$\{\delta^j(\pi_\gamma(\psi_i))\}$ forms a frame for $\mathcal H,$ then
$\delta$ can have no eigenvector.  Indeed, if $v$ were an
eigenvector for $\delta,$ then the numbers
$|\langle v\mid \delta^j(\pi_\gamma(\psi_i))\rangle|$ are constant independent of $j.$
Hence, they must all be 0, contradicting the frame assumption.
Therefore, in the original context of the trivial intersection problem,
i.e., where the subspaces $V_j = \overline{\mbox{span}}\{\delta^k(\pi_\gamma(\psi_i)): \ k<j, \ \gamma, \ i\}$ are constructed from a Parseval wavelet $\{\psi_i\},$
there is no eigenvector for $\delta.$  Hence, if the intersection is nontrivial,
then it must be infinite-dimensional,
and the multiplicity function $m$ is infinite almost everywhere.  Theorem \ref{intersection} therefore extends Theorem 6.1 of \cite{BR} from the classical case of $\H=L^2(\R^d)$ to an abstract Hilbert space.

In the case when $\H=L^2(\R^d)$, the group $\Gamma = \Z^d$ acts by translation, and $\delta(f) = \delta_A(f) = |\det A|^{1/2} f(A\cdot)$ for an expansive integer matrix $A$, we need not require the subspaces $\{V_j\}$ to be constructed from a Parseval wavelet, as in the previous paragraph.  For any bounded neighborhood $E\subseteq \R^d$ of 0, there exists a positive integer $n$ such that $E\subset A^nE$. Since
\begin{equation*}
\int_E |f(x)|^2 \ dx = \int_{A^kE} |f(x)|^2 \ dx
\end{equation*}
for any $f$ that satisfies $\overline\lambda f(x) = |\det A|^{1/2} f(Ax)$, we see that $\delta_A$ can have no eigenvector, and, by Theorem \ref{intersection}, the assumption that $m$ is not identically $\infty$ a.e. implies the trivial intersection property.
  Thus, Theorem \ref{intersection} provides an alternative proof for the classical scenario of Theorem 1.1 of \cite{Bow}.

\end{document}